\documentclass[12pt]{article}
\usepackage{amssymb,amsmath,amsfonts,mathrsfs,amsthm,epsfig,latexsym,color}
\usepackage{enumitem,geometry}
\usepackage{fourier}
\usepackage[all]{xy}
\usepackage{setspace}
\usepackage{indentfirst}
\makeatletter
\newcommand{\subsectionruninhead}{\@startsection{subsection}{2}{0mm}
{-\baselineskip}{-0mm}{\bf\large}}
\newcommand{\subsubsectionruninhead}{\@startsection{subsubsection}{3}{0mm}
{-\baselineskip}{-0mm}{\bf\normalsize}}
\makeatother

\newtheorem*{theorem*}{Theorem}
\newtheorem*{proof*}{Proof}
\newtheorem*{proposition*}{Proposition}
\newtheorem*{corollary*}{Corollary}
\newtheorem*{claim*}{Claim}
\newtheorem*{remark*}{Remark}

\newtheorem{theorem}{Theorem}[section]
\newtheorem{proposition}{Proposition}[section]

\newtheorem{corollary}[proposition]{Corollary}
\newtheorem{lemma}[proposition]{Lemma}

\theoremstyle{definition}
\newtheorem{definition}[proposition]{Definition}
\theoremstyle{remark}
\newtheorem{remark}[proposition]{Remark}

\newtheorem{question}[proposition]{Question}
\numberwithin{equation}{section}

\def\NN{\mathbb{N}}
\def\RR{\mathbb{R}}
\def\CC{\mathbb{C}}

\def\TT{\mathbb{T}}
\def\ZZ{\mathbb{Z}}
\def\QQ{\mathbb{Q}}
\def\d0{\delta^{(0)}}
\def\orb{\mathrm{Orb}}
\def\modd{\mathrm{mod}}

\begin{document}
\title{Uniformly distributed periodic orbits of endomorphisms on $n$-tori}
\author{Daohua Yu and Shaobo Gan}
\date{\today}
\maketitle
\begin{abstract}
We prove that any ergodic endomorphism on an $n$-torus admits a sequence of periodic orbits uniformly distributed in the metric sense. As a corollary, an endomorphism on the torus is ergodic if and only if the Haar measure can be approximated by periodic measures.
\end{abstract}
\section{Introduction}
\indent For any natural number $n\ge 1$, denote by $M_n(\ZZ)$ the set of $n\times n$ matrices with entries in $\ZZ$. Any $A\in M_n(\ZZ)$ can be naturally considered as a linear map from $\RR^n$ to $\RR^n$, i.e., $v\mapsto Av$; since $A(\ZZ^n)\subset \ZZ^n$, $A$ also induces a map from the $n$-torus $\TT^n=\RR^n/\ZZ^n$ to itself, i.e., $v\mapsto Av \mod 1$, still denoted by $A$. If $\det(A)\not=0, $ such a map will be called an {\em endomorphism} on the $n$-torus. The best known endomorphism may be Arnold's cat map on the 2-torus, which is induced by $A=\left[\begin{array}{cc} 2&1\\ 1&1\end{array}\right].$ Related to the sophisticated construction of Birkhoff sections in his thesis \cite{Shi}, Y. Shi raised the following question:

\begin{question}[Shi's question] Does Arnold's cat map admit a sequence of periodic orbits $\{O_k\}_{k\in\NN}$ such that
$$
d(O_k)^2 T(O_k)\ge C
$$
for some constant $C>0$ $?$ Here $T(O_k)$ denotes the period of $O_k$ and
$$
d(O_k)=\min\{d(x,y): x, y\in O_k, x\not=y\}.
$$
\end{question}

In this paper, we will give an affirmative answer to Shi's question. In fact, we will prove a slightly general result, i.e., similar periodic orbits exist for any ergodic torus endomorphism. Here, ``ergodic'' means that the endomorphism is ergodic with respect to Leb, the Haar measure on the torus.

\begin{theorem}\label{metric}
Let $A$ be an ergodic endomorphism on an $n$-torus. Then there is a positive number $C$  and a sequence of periodic orbits $\{O_k\}_{k\in\NN}$, such that
$$ d(O_k)^nT(O_k)\ge C \qquad \text{and}\qquad \lim_{k\to +\infty}T(O_k)=+\infty.\eqno (*)$$
\end{theorem}

\begin{remark}
For any periodic orbit $O$ with period $T$, according to the definition of $d(O)$ we know that $\bigcup_{x\in O} B(x,\frac{d(O)}{2})$ are pairwise disjoint. We have
$$
T(O)\left(\frac{d(O)}{2}\right)^n\omega_n\leq 1,
$$
where $\omega_n$ is the volume of $n$ dimensional unit ball.
So there is a constant $C'=\frac{2^n}{\omega_n}$ independent of $O$ such that $d(O)^nT(O)\leq C'$. Thus, $n$ is the largest exponential index such that the equation $(*)$ holds.
\end{remark}

\begin{definition}\label{UD}
A sequence of periodic orbits $\{O_k\}_{k\in\NN}$ satisfying $(*)$ will be called {\em uniformly distributed} (in the metric sense).
\end{definition}

Recall that a sequence of points $\{x_k\}_{k\in\NN}$ in an $n$-torus is called {\em uniformly distributed} if $$\frac{1}{k}\sum_{i=1}^k\delta_{x_i} \xrightarrow{\text{weak*}} Leb\quad (k\to +\infty).$$  The following proposition justifies  Definition \ref{UD}.

\begin{proposition}\label{transfer metric to measure}
Let $A$ be an ergodic endomorphism on an $n$-torus and $\{O_k\}_{k\in\NN}$ be a sequence of periodic orbits of $A$.
Denote the periodic measure $\mu_k=T(O_k)^{-1}\sum_{x\in O_k}\delta_x$.  If $\{O_k\}_{k\in\NN}$ is uniformly distributed, then  $\mu_k$ weak* converges to Leb.
\end{proposition}

As a corollary, we obtain the following periodic approximation result.
\begin{corollary}\label{iff}
An endomorphism  $A$ on the torus is ergodic if and only if $A$ has a sequence of periodic measures weak* converging to Leb.
\end{corollary}

\begin{remark}
\begin{enumerate}
\item Started from \cite{A}, there are lots of research on periodic approximation, e.g., \cite{Sigmund,Katok, LST, DGS}, etc..
\item The quantity $d(O)$ also appears in the study of ergodic optimization (\cite{HLMXZ}).
\item According to the proof of Proposition \ref{transfer metric to measure}, any limit of $\mu_k$ is absolutely continuous w.r.t. Leb. Hence, small perturbations of ergodic endomorphisms on the torus may not admit a sequence of uniformly distributed periodic orbits.
\end{enumerate}
\end{remark}

Here, we provide an outline of the proof of Theorem \ref{metric}, especially the construction of the sequence of uniformly distributed periodic orbits, which is the core of this paper.
The endomorphism on an $n$-torus can be classified into two cases, one is irreducible, the other is reducible.
For the irreducible case,  periodic orbits will be chosen from the eigenvectors of $A$ modulo powers of some primes. These primes come from Frobenius density theorem in number theory.
By comparing the period of the eigenvector with the period of the linear recurrence sequence induced by $A$ with the initial condition $0,\cdots,0,1$, we can see that one of eigenvectors has a big period, and this one is the periodic point picked out. For the reducible case, the structure theorem for modules over polynomial rings is used to simplify the problem.

\begin{question}
Does any small perturbation of an expanding endomorphism $L$ on circle have uniformly distributed periodic orbits? If the answer is no, assume that $f$ is a small perturbation of $L$ which has uniformly distributed periodic orbits, whether $f$ is smoothly conjugate to $L$ or not?
We also curious about the phenomenon when the system is a small conservative perturbation of a hyperbolic endomorphism on the torus.
\end{question}

\noindent \textbf{Acknowledgement.} We wish to thank Yi Shi for his interesting question. We would like to thank the referees for many nice comments and suggestions. This work is partially supported by National Key R\&D Program of China 2022YFA1005801 and NSFC 12161141002.

\section{Preliminary}

In this section we list some basic results from number theory, including Lifting the Exponent Lemma (Theorem \ref{LEL}) and Frobenius Density Theorem (Theorem \ref{FDT}).

\begin{theorem}\cite{ADM}\label{LEL} (Lifting the Exponent Lemma)
Let $p$ be an odd prime, $x,y,k,m\in \ZZ, k>0$.
Denote $v_p(m)=\max\{i\in\ZZ_{\geq 0}:p^i\mid m\}$.
Suppose that $p\mid x-y$ but $p\nmid x, p\nmid y$. Then
$$v_p(x^k-y^k)=v_p(x-y)+v_p(k).$$
\end{theorem}

\begin{corollary}\label{order}
Given $a\in \ZZ, |a|\neq 1$, and an odd prime $p$, $(a,p)=1$, let $d=\min\{n>0: p\mid a^n-1\}$ and $v_p(a^d-1)=t$.
Then for any $k\geq t, k\in \ZZ$, $p^k\mid a^n-1$ if and only if $p^{k-t}d\mid n$.
\end{corollary}
\begin{proof}
Since $p\mid a^n-1$, $n=dl, l\in \ZZ$.
By using Theorem \ref{LEL} to $x = a^d$, $y=1$, we know that $v_p(a^n-1)=v_p(x^l-1)=v_p(x-1)+v_p(l)=t+v_p(l)$. Thus $v_p(a^n-1)\geq k$ if and only if $v_p(l)\geq k-t$, if and only if $p^{k-t}d\mid n$.
\end{proof}

\begin{theorem}\label{FDT}(\cite[Corollary 2]{density2}) ({\rm Frobenius Density Theorem}) Let $A\in M_n(\ZZ).$ Denote by $f(x)=\det(xI-A)=x^n-c_{n-1}x^{n-1}-\cdots -c_1x-c_0$. Assume that $f(x)$ is irreducible on $\QQ$. The set of primes $p$ for which $f$ modulo $p$ splits completely into linear factors has positive density, i.e.,
$\Lambda$ has positive density in primes, where
$$
\Lambda=\{p\in \NN : p \text{is a prime},\, \exists a_1,a_2,\dots,a_n\in\ZZ, s.t., f(x)\equiv\prod_{i=1}^n(x-a_i)\mod p\}.
$$
\end{theorem}

\begin{corollary}\label{|xI-A|}
Under the assumptions of Theorem \ref{FDT}, there are infinitely many primes $p$ such that $f(x)\equiv 0\mod p$ has $n$ nonzero distinct integer solutions $\{a_i\}_{i=1}^n$, i.e.,
$$f(a_i)\equiv 0\mod p,\quad a_i\not\equiv a_j\mod p,\  i\neq j,\quad a_i\not\equiv 0\mod p.$$
Furthermore, for such $p$, any $k\in\NN$, $f(x)\equiv 0\mod p^k$ has $n$ nonzero distinct integer solutions $\{b_i\}_{i=1}^n$, i.e.,
$$f(b_i)\equiv 0\mod p^k,\quad b_i\not\equiv b_j\mod p,\ i\neq j,\quad b_i\not\equiv 0\mod p.$$
\end{corollary}

\begin{proof}
According to \cite[P191, Theorem 6.1]{Lang}, there is an integral polynomial $h(x_0,x_1,\dots,x_{n-1})$ such that the discriminant $D(g)$ of any $g\in\ZZ[x]$ satisfies $D(g)=h(b_0,b_1,\dots, b_{n-1})$, where $g(x)=x^n-b_{n-1}x^{n-1}-\cdots -b_1x-b_0$, $D(g)=\prod_{1\leq i<j\leq n}(y_i-y_j)^2$,
$g(y_i)=0, y_i\in\CC, 1\leq i\leq n.$
From Theorem \ref{FDT}, we know that there are infinitely many primes $p$ larger than both $|D(f)|$ and $|c_0|$,  s.t., $f(x)\equiv r(x)\mod p$, where $r(x)=\prod_{i=1}^n(x-a_i)=x^n-e_{n-1}x^{n-1}-\cdots -e_1x-e_0$, $e_i,a_i\in\ZZ$. Then $$D(f) = h(c_0,c_1,\dots, c_{n-1})\equiv h(e_0,e_1,\dots, e_{n-1})=D(r)=\prod_{1\leq i<j\leq n}(a_i-a_j)^2\mod p.$$
Thus
$(p,D(f))=1$ implies $a_i\not\equiv a_j \mod p$, for any $i\neq j$. In other words, $f'(a_i)\not\equiv 0 \mod p$, $1\leq i\leq n$. Since $(p,c_0)=1$, $a_i\not\equiv 0\mod p$, $1\leq i\leq n$. This finishes the proof of the first part of the corollary.\\
\indent Suppose that there is a solution $x\equiv x_k\mod p^k$ satisfying that $x_k\not\equiv 0\mod p$, $f(x)\equiv 0\mod p^k$ and $f'(x)\not\equiv 0\mod p$. Thus one can write $f(x_k)$ as $p^lr$, $p\nmid r$ and $l\geq k$. Suppose $t\equiv s\mod p$ is the unique solution of $$f'(x_k)t+r\equiv 0\mod p.$$
Then, let $x_{k+1}\equiv x_k+p^ls\mod p^{k+1}$. Now we have $x_{k+1}\equiv x_k\not\equiv 0\mod p$. Since
$$f(x_{k+1})\equiv f(x_k)+f'(x_k)sp^l\equiv p^l(r+f'(x_k)s)\equiv 0 \mod p^{k+1},$$
$x\equiv x_{k+1}$ is a solution of $f(x)\equiv 0\mod p^{k+1}$ and $f'(x)\not\equiv 0\mod p$.
The conclusion follows by induction.
\end{proof}

\begin{remark}\label{eigenvalue}
Let $A\in M_n(\ZZ)$ and $k\in\NN$. If $f(x)=\det(xI-A)=x^n-c_{n-1}x^{n-1}-\cdots -c_1x-c_0\equiv 0\mod p^{k}$ has a solution $x\equiv a\mod p^k$, then for $v=(1,a,a^2,\cdots, a^{n-1})^T\in\ZZ^n$ we have $Bv\equiv av\mod p^k$, where $B$ is the {\em companion matrix} of $f(x)$, i.e.,
$$B=\left(\begin{matrix}0&1&&&\\&0&1&&\\\cdots&\cdots&\cdots&\cdots&\cdots\\&&&0&1\\c_0&c_1&\cdots&\cdots &c_{n-1}\end{matrix}\right).$$
Here ``$T$'' means transpose.

\end{remark}

\begin{lemma}\label{invertible}
Given $e\in\ZZ^n, e\neq 0$, $A\in M_n(\ZZ)$, denote
$$P=P(e, A)=\left(\begin{matrix}e\\eA\\\cdots\\eA^{n-1}\end{matrix}\right).\eqno{(P)}$$
If $f(x)=\det(xI-A)$ is irreducible on $\QQ$, then $P$ is invertible.
\end{lemma}
Here there is an abuse of notation: ``$v\in\ZZ^n$'' can be either a row vector or a column vector according to the context.
\begin{proof}
Denote $$S=\{d(x)\in \QQ[x]:ed(A)=\vec{0}\}.$$
Suppose that $g(x)\in S$ satisfies that $\deg g(x)=\min_{d\in S}\deg d(x)$, where $\deg$ represents the polynomial degree. We claim that $S=(g(x))=\{g(x)a(x):a(x)\in \QQ[x]\}$. For any $h(x)\in S$, there are $s(x), r(x)\in \QQ[x]$ such that $h(x)=s(x)g(x)+r(x)$ and $\deg r(x)<\deg g(x)$. Since $$
er(A)=eh(A)-eg(A)s(A)=\vec{0}\quad \text{and}\quad \deg r(x)<\deg g(x),$$
by the definition of $g(x)$, we know that $r(x)=0$, i.e., $h(x)\in (g(x))$.\\
\indent Since $f(A)=0$, $f(x)\in S=(g(x))$. While $f(x)$ is irreducible, it implies that $f(x)=cg(x)$, $c\in\QQ$. So $e,eA,\cdots, eA^{n-1}$ must be linearly independent over $\QQ$, otherwise there will be a polynomial with degree strictly smaller than $n$ in $S$. Thus $\det(P)\neq 0$, i.e., $P$ is invertible.
\end{proof}

\begin{definition}
Given $f_i\in\ZZ, i=1,2,\dots,k$,
a sequence $\{u_m\}_{m\geq 0}$ of integers defined by
\begin{equation}\label{LRS}
u_{m+k}=\sum_{j=1}^kf_ju_{m+k-j}
\end{equation}
is called a linear recurrent sequence with the associated polynomial
$$F(x)= x^k-f_1x^{k-1}-\cdots -f_{k-1}x-f_k.$$
\end{definition}

Recall  the notation introduced by Dedekind. Two integer polynomials $a(x)$ and $b(x)$ satisfy $a(x)\equiv b(x) \ (\modd\ m, c(x))$, if and only if $a(x)-b(x)$ is divisible by $c(x)$ modulo $m$, i.e., $a(x)-b(x)=c(x)q(x)+mr(x)$ with integer polynomials $q(x)$ and $r(x)$.

\begin{theorem}\label{FPS}
\cite[page 606]{recurrence} (Fundamental theorem on purely periodic sequences)\\
\indent Let $q$ be an integer larger than $1$. If $\{u_m\}_{m\geq 0}$ is a linear recurrent sequence defined by (\ref{LRS}), with the associated polynomial $F(x)$, then a necessary and sufficient condition that
$\{u_m \modd\ q\}_{m\geq 0}$ should be purely periodic and admit the period $T$ is that
$$(x^T-1)U(x)\equiv 0\quad (\modd\ q, F(x))$$
where
$$U(x)=u_0x^{k-1}+(u_1-f_1u_0)x^{k-2}+\cdots+(u_{k-1}-f_1u_{k-2}-\cdots-f_{k-1})$$
is a polynomial of degree $k-1$ in $x$ whose coefficients are determined entirely by the $k$ initial values of $\{u_m\}_{m\geq 0}$ and the coefficients $f_j$ of (\ref{LRS}).
\end{theorem}

\begin{definition}\label{period of LRS}
Given $A\in M_n(\ZZ)$, denote $f(x)=\det(xI-A)= x^n-c_{n-1}x^{n-1}-\cdots -c_1x-c_0$. The linear recurrence sequence $\{u_k\}$ defined by
\begin{equation}\label{LRS of A}
u_{n+k}=c_{n-1}u_{n+k-1}+\cdots +c_1u_{k+1}+c_0u_k,\qquad u_0=u_1=\cdots =u_{n-2}=0, u_{n-1}=1
\end{equation}
is called the {\em linear recurrence sequence induced by $A$}. For any prime number $p$ such that $(p,\det(A))=1$, denote by $T_r$ the period of sequence $\{u_k\mod p^r\}$.
\end{definition}

Similar to \cite[Corollary 3]{Bright}, the following result gives the period of a linear recurrence sequence module a power of prime which is proved here by another method.

\begin{lemma}
Given $A\in M_n(\ZZ)$, $T_r$ is defined in Definition \ref{period of LRS}, and $p$ is an odd prime. If the induced endomorphism of $A$ on the $n$-torus is ergodic, then there is $t\in\NN$ such that $T_1=\cdots=T_t$ and $T_k=p^{k-t}T_1$ for any $k>t$.
\end{lemma}

\begin{proof}
\indent $\{u_k \mod p^r\}$ is periodic since $(p,\det(A))=1$. According to Theorem \ref{FPS}, and noting $U(x)\equiv 1$ now, the period $T_r$ of $\{u_k \mod p^r\}$ is the minimal positive integer $k$ such that
$$x^k\equiv 1 \quad (\modd\ p^r,f(x)),$$
i.e., $$x^k-1=f(x)q(x)+p^rt(x),\quad q(x),t(x)\in \ZZ[x].$$
Since $A$ is an ergodic endomorphism on an $n$-torus, any root of unity is not a root of $f(x)$. Therefore there is positive integer $t<+\infty$ and $p\nmid r(x)$ with $\deg(r(x))<\deg(f(x))$ such that $x^{T_1}\equiv 1+p^tr(x)\mod f(x)$. Then $T_1=\cdots=T_t$. Notice the fact if $x^k\equiv 1+p^ms(x)\mod f(x)$ with $p\nmid s(x)$ and $m\geq 1$ then
$$x^{pk}\equiv (1+p^ms(x))^p\equiv 1+p^{m+1}s(x)\quad (\modd\ p^{m+2},f(x)),$$
i.e.
$$x^{pk}\equiv 1+p^{m+1}l(x)\mod f(x),\quad \text{with}\ p\nmid l(x).$$
By induction, for any $i\in\NN$,
$$x^{p^iT_1}\equiv 1+p^{t+i}r_i(x)\mod f(x), \quad \text{with}\ p\nmid r_i(x).$$
It implies that $T_{t+i}\mid p^iT_1$ and $T_{t+i}\nmid p^{i-1}T_1$. Combining with $T_1\mid T_{t+i}$, we get that  $T_{t+i}=p^{i}T_1$ for any $i\in\NN$.
\end{proof}

\section{Proof of main results}

\subsection{Proof of Theorem 1.1}
\begin{lemma}\label{similarity-matrix}
Assume that $A\in M_n(\mathbb{Z})$ induces an ergodic endomorphism on an $n$-torus. Moreover, assume that there exist $P, B\in M_n(\ZZ)$, $\det(P)\neq 0$ such that $PA =BP$. If $B$ admits a sequence of uniformly distributed periodic orbits, then $A$ admits a sequence of uniformly distributed periodic orbits.
\end{lemma}

\begin{proof}
For any periodic orbit $O$ of $B$, one has $A(P^{-1}O)\subset P^{-1}(O)$. Hence there exists $O'\subset P^{-1}(O)$, which is a periodic orbit of $A$, and $T(O')=kT(O)$, $1\le k\le |\det(P)|$, $k\in \NN$.
Since $P$ is a local diffeomorphism, $d(O')\geq ||P||^{-1}d(O)$.\\
\indent Choose a sequence of periodic orbits $\{O_k\}_{k\in\NN}$ such that $d(O_k)^nT(O_k)\geq C$ with constant $C$ and $\lim_{k\to +\infty}T(O_k)=+\infty$. Then $d(O'_k)^nT(O'_k)\geq ||P||^{-n}C$ and
$\lim_{k\to +\infty}T(O'_k)=+\infty$.
\end{proof}

Firstly, we deal with the irreducible case of Theorem 1.1. Recall that a
monic irreducible polynomial $f(x)$ on $\QQ$ means that
$f(x)=x^k+\sum_{i=0}^{k-1}a_ix^i, a_i\in\QQ$, and any $g(x)\mid f(x)$ implies that $g(x)$ is $c$ or $cf(x)$, where $c$ is a rational constant.
\begin{proposition}\label{n-irreducible}
Assume $A$ is an ergodic endomorphism on an $n$-torus.
If $f(x)=\det(xI-A)=x^n-c_{n-1}x^{n-1}-\cdots -c_1x-c_0$ is irreducible on $\QQ$, then $A$ admits a sequence of uniformly distributed periodic orbits.
\end{proposition}

\begin{proof}
If $n=1$, $Ax=c_0x\mod 1$ is the induced endomorphism on $S^1$, with $c_0\neq 1,-1$. Choose odd prime $p$ such that $(p,c_0)=1$. Let $v_k=p^{-k}$, $O_k = \orb^+(v_k)$. From Corollary \ref{order} we know that $d(O_k)\geq p^{-k}$, and
$$T(O_k)=\min\{n>0: p^k\mid (c_0)^n-1\}\geq p^{k-t}, \,t=v_p((c_0)^d-1),\, d=\min\{n>0:p\mid (c_0)^n-1\},$$
thus $d(O_k)T(O_k)\geq p^{-t}$. The conclusion holds.\\
\indent Now assume  $n>1$. Take $e\in\ZZ^{n}, e\not=0$ and let $P=P(e,A)$ be defined by the formula $(P)$.
Then $PA=BP$, where $B$ is the companion matrix of $f(x)$. By Lemma \ref{invertible} and Lemma \ref{similarity-matrix}, it is enough to prove the proposition for $B$.\\
\indent Take an odd prime $p$ satisfying Corollary \ref{|xI-A|}. Then $f(x)\equiv 0\mod p^k$ has solutions $x\equiv b_i\mod p^k$ and $b_i\not\equiv b_j\mod p$ for $i\neq j$.
Denote $w_i=(1,b_i,b_i^2,\cdots,b_i^{n-1})^T$, then $Bw_i\equiv b_iw_i\mod p^k$. Take $v^{(i)}=p^{-k}w_i$, $1\leq i\leq n$.
Consider the forward orbit $O_k^t$ of $v^{(t)}$. By Euler's Theorem, $$B^{p^{k-1}(p-1)}w_i\equiv b_i^{p^{k-1}(p-1)}w_i\equiv w_i\mod p^k.$$ Hence $O_k^t$ is a periodic orbit.
Given $w\in\ZZ^n,$ define $$I(w)=w\wedge Bw\wedge\cdots\wedge B^{n-1}w =\det (w \, Bw \, \cdots B^{n-1}w).$$
Denote $a= (a_1,a_2,\cdots,a_n)^T \equiv B^iw_t-B^jw_t \mod p^k$ such that $|a_i|<\frac{p^k}{2}$, $1\leq i\leq n$.
\begin{equation*}
\begin{split}
B^sa &= B^s(B^iw_t-B^jw_t+p^kC_1)\\
&=B^s(B^i-B^j)w_t+p^kC_2\\
&=(B^i-B^j)(b_t)^sw_t+p^kC_3\\
&=(b_t)^s(B^iw_t-B^jw_t)+p^kC_3\\
&=(b_t)^sa+p^kC_4
\end{split}
\end{equation*}
where $C_i\in \ZZ^n$, $1\leq i\leq 4$.

 Then $$
I(a)=a \wedge (b_ta+p^km_1)\wedge \cdots \wedge ((b_t)^{n-1}a+p^km_{n-1})=a \wedge p^km_1\wedge\cdots\wedge p^km_{n-1}\equiv 0\mod p^{k(n-1)}
$$
where $m_i\in\ZZ^n$, $1\leq i\leq n-1$. From Lemma \ref{invertible}, we know that $I(a)\neq 0$, so $|I(a)|\geq p^{k(n-1)}$.
On the other hand, we have
$$|I(w)|=|w\wedge Bw\wedge \cdots\wedge B^{n-1}w|\leq |w||Bw|\cdots|B^{n-1}w|\leq \prod_{i=1}^{n-1}||B^i|||w|^n.$$
So $$|a|^n\geq C_0|I(a)|\geq C_0p^{(n-1)k}.$$
We get the estimation of $d(O^t_k)$ from below: $$d(O^t_k)^n\geq C_0\frac{p^{(n-1)k}}{p^{nk}}=C_0p^{-k}.$$

Now we estimate the period $T^{(t)}_k$ of $O^t_k$.
Consider the linear recurrence sequence $\{v_i\}$ satisfying that
$$v_{i+n}=c_{n-1}v_{i+(n-1)}+c_{n-2}v_{i+n-2}+\cdots +c_0v_i.\eqno {(**)}$$
Denote by $\{v_i\}$ the (general) solution of $(**)$ with the initial condition $v = (v_0, v_1, \cdots , v_{n-1})$ and $\{u_i\}$ the solution with the initial condition $(u_0,\cdots, u_{n-2}, u_{n-1})=(0,\cdots,0,1)$.  In the following, we will write the general solution $\{v_i\}$ as a linear combination of the special solution $\{u_i\}$.

Let $x=(x_1,x_2,\cdots,x_n)$ be the solution of $xQ=v$,
where
$$
Q=\left(\begin{matrix}
0&0&\cdots&0&0&1\\
0&0&\cdots&0&1&u_n\\
0&0&\cdots&1&u_n&u_{n+1}\\
\cdots&\cdots&\cdots&\cdots&\cdots&\cdots\\
1&u_n&\cdots&u_{2n-4}&u_{2n-3}&u_{2n-2}
\end{matrix}\right)=\left(\begin{matrix}
-c_1&-c_2&\cdots&-c_{n-2}&-c_{n-1}&1\\
-c_2&-c_3&\cdots&-c_{n-1}&1&0\\
\cdots&\cdots&\cdots&\cdots&\cdots&\cdots\\
-c_{n-1}&1&\cdots&0&0&0\\
1&0&\cdots&0&0&0
\end{matrix}\right)^{-1}.
$$
Then
\begin{equation*}
\begin{split}
v_m =& (x_1,x_2,\cdots,x_n)(u_m,u_{m+1},\cdots,u_{m+n-1})^T\\
=&(v_0, v_1, \cdots, v_{n-1})
Q^{-1}(u_m,u_{m+1},\cdots,u_{m+n-1})^T\\
=& vQ^{-1}U_m
\end{split}
\end{equation*}
where $$U_m = (u_m,u_{m+1},\cdots,u_{m+n-1})^T.$$
Since $\{b_t^m\}$ is the solution of $(**) \mod p^k$ with the initial condition $v=(1,b_t,\cdots, b_t^{n-1})$, one has $b_t^m\equiv (1,b_t,\cdots, b_t^{n-1})Q^{-1}U_m\mod p^k$. It means that $T^{(t)}_k\mid T^u_k$, where $T^u_k$ is the period of $u_n\mod p^k$.

On the other hand, we have
$$
\left(\begin{matrix}b_1^m\\b_2^m\\\cdots\\b_n^m\end{matrix}\right)\equiv
\left(\begin{matrix}
1&b_1&\cdots&b_1^{n-1}\\1&b_2&\cdots&b_2^{n-1}\\\cdots&\cdots&\cdots&\cdots\\1&b_n&\cdots&b_n^{n-1}
\end{matrix}\right)
Q^{-1}
\left(\begin{matrix}u_{m}\\u_{m+1}\\\cdots\\u_{m+n-1}\end{matrix}\right)\mod p^k.
$$

Since
$$\det \left(\begin{matrix}1&b_1&\cdots&b_1^{n-1}\\1&b_2&\cdots&b_2^{n-1}\\\cdots&\cdots&\cdots&\cdots\\1&b_n&\cdots&b_n^{n-1}
\end{matrix}\right)\equiv \prod_{i<j}(b_j-b_i)\not\equiv 0\mod p,$$

we have that
$$
\left(\begin{matrix}1&b_1&\cdots&b_1^{n-1}\\1&b_2&\cdots&b_2^{n-1}\\\cdots&\cdots&\cdots&\cdots\\1&b_n&\cdots&b_n^{n-1}\end{matrix}\right)^{-1}
\left(\begin{matrix}b_1^m\\b_2^m\\\cdots\\b_n^m\end{matrix}\right)\equiv
Q^{-1}
\left(\begin{matrix}u_{m}\\u_{m+1}\\\cdots\\u_{m+n-1}\end{matrix}\right)\mod p^k.
$$
Thus $u_m \equiv r_1b_1^m+r_2b_2^m+\cdots+r_nb_n^m\mod p^{k}$ for some $r_1, r_2, \cdots, r_n\in\ZZ$. It implies that the period $T^u_k$ of $u_n\mod p^k$ is a factor of $[T^{(1)}_k,T^{(2)}_k,\cdots, T^{(n)}_k]$. Combining with $T^{(i)}_k\mid T^u_k$, we have that $[T^{(1)}_k,T^{(2)}_k,\cdots, T^{(n)}_k]=T^u_k$. It means that there is $i$ such that $$
v_p(T^{(i)}_k)=v_p(T^u_k)=k-t.$$
So there is a positive number $C$ such that $d(O^i_k)^nT^{(i)}_k\ge C$.
\end{proof}

Now, suppose that $f(x)=\det(I-xA)$ is reducible on $\QQ$. From Gauss Lemma we know that there are two integer polynomial $g(x)$ and $h(x)$ such that $f(x)=g(x)h(x)$. Furthermore $f(x)=f_1(x)^{\alpha_1}f_2(x)^{\alpha_2}\cdots f_r(x)^{\alpha_r}$, where $f_i(x)(1\leq i\leq r)$ are different integer polynomials and irreducible on $\QQ$, and $\alpha_i\in\NN(1\leq i\leq r)$.\\
\indent Denote $|v\mod 1|=\min \{|v-n|:n\in \ZZ^k\}$ for $v\in\RR^k$.
\begin{proposition}
Assume $A$ is an ergodic endomorphism on an $n$-torus. If the minimal polynomial of $A$ is the characteristic polynomial $\det(xI-A)=f(x)=x^n-c_{n-1}x^{n-1}-\cdots -c_1x-c_0$, and $f(x)=g(x)^r$, $g(x)=x^m-d_{m-1}x^{m-1}-\cdots -d_1x-d_0$ is an integer polynomial and irreducible on $\QQ$,
then $A$ admits a sequence of uniformly distributed periodic orbits.
\end{proposition}
\begin{proof}
Let $D$ be the companion matrix of $g(x)$. Denote $$B=\left(\begin{matrix}D&I&0&\cdots&0&0\\0&D&I&\cdots&0&0\\\cdots&\cdots&\cdots&\cdots&\cdots&\cdots
\\0&0&0&\cdots&D&I\\0&0&0&\cdots&0&D\end{matrix}\right)\in M_n(\ZZ),$$
and $E={\rm diag}\{D,D,\cdots,D\}$ with $r$ numbers of $D$. Then $B=E+N, EN=NE, N^r=0$.
Then for any polynomial $h(x)$, 
$$h(B)=h(E)+h'(E)N+\cdots +((r-1)!)^{-1}h^{(r-1)}(E)N^{r-1}.\eqno(T)$$
Since the minimal polynomial of $D$ is $g$ and $(g,g')=1$, we have $(g^i)^{(i)}(D)\neq 0$ for $1\leq i\leq r-1$.
Letting $h=g^i$ in the equation $(T)$, we get $g^i(B)\neq 0$ for $1\leq i\leq r-1$.
Combining with $\det(xI-B)=\det(xI-D)^r=g^r(x)=f(x)$, we know that the minimal polynomial of $B$ is $f(x)$.
So there are $\alpha\in\ZZ^n\cap(\QQ^n-\ker(g^{r-1}(A)))$ and $\beta\in\ZZ^n\cap(\QQ^n-\ker(g^{r-1}(B)))$. Then 
$\{s(x)\in \ZZ[x]:\alpha s(A)=0\}=\{s(x)\in \ZZ[x]:\beta s(B)=0\}=(g^r(x))=(f(x))$.
Hence
$\{\alpha A^i\}_{i=0}^{n-1}$ and $\{\beta B^i\}_{i=0}^{n-1}$ are linearly independent over $\QQ$ respectively. Denote
$$Q=P(\alpha, A)\quad R =P(\beta, B).
$$
Then we have
$$QA = A_*Q\quad \text{and}\quad RB=A_*R, $$
where $ A_*$ is the companion matrix of $f(x)$.
So $R^{-1}QA=R^{-1}A_*Q=BR^{-1}Q$, by multiplying a constant, we get that there is $P\in M_n(\ZZ), \det (P)\not=0$ such that $PA=BP$.\\
\indent Take primes $p_1>p_2>\cdots >p_r$ satisfying the requirement in Corollary \ref{|xI-A|} for $g(x)=\det(xI-D)$. Let
$$k_1=k, k_{i+1}=\left[\frac{k_i\ln p_i}{\ln p_{i+1}}\right]\quad \Rightarrow\quad C^{-1}=\prod_{1\leq t\leq r}p_t^{-1}<\frac{p_i^{k_i}}{p_j^{k_j}}<\prod_{1\leq t\leq r}p_t=C$$
and
if $k_1'=k+1, k_{i+1}'=\left[\frac{k_i'\ln p_i}{\ln p_{i+1}}\right]$, we have
$$
\prod_{1\leq i\leq r}p_i^{k_i'}\leq (p_1^{k})^rp_1^r\leq \prod_{1\leq i\leq r}p_i^{k_i}p_1^rC^r.
$$
This will be used later in the proof of Theorem \ref{metric}, where $$m_k=\prod_{1\leq i\leq r}p_i^{k_i}\qquad\text{and}\qquad m_{k+1}=\prod_{1\leq i\leq r}p_i^{k_i'}.$$
There are $w_i\in\ZZ^m-\{0\}$ and $b_i\in\ZZ$ such that
$$Dw_i\equiv b_iw_i\mod p_i^{k_i}, \quad p_i^{k_i}\mid C_iT(\orb^+(p_i^{-k_i}w_i)), \quad d(\orb^+(p_i^{-k_i}w_i))^m\geq C_i'^{-1}p_i^{-k_i},$$
where $C_i, C_i'$ are constants independent of $k$. Let $$v = (p_1^{-k_1}w_1, p_2^{-k_2}w_2,\cdots,p_r^{-k_r}w_r).$$
Since $$B^j=\left(\begin{matrix}D^j&C^1_jD^{j-1}&\cdots&C^{r-1}_jD^{j-r+1}\\0&D^j&\cdots&C^{r-2}_jD^{j-r+2}\\
\cdots&\cdots&\cdots&\cdots\\0&0&\cdots&D^j\end{matrix}\right)\quad \text{and}\quad C^i_j=\frac{j}{i}C^{i-1}_{j-1},$$
we have that $B^{\prod_{1\leq i\leq r}iC_iT(\orb^+(p_i^{-k_i}w_i))}v\equiv v\mod 1$, i.e., $\orb^+(v)$ is a periodic orbit.\\
\indent Suppose that $B^lv\equiv v\mod 1$. Since $(p_i,p_j)=1$ for $i\neq j$, we have $D^l(p_i^{-k_i}w_i)\equiv p_i^{-k_i}w_i \mod 1$ for $1\leq i\leq r$. Thus $T(\orb^+(p_i^{-k_i}w_i))\mid l$, $\prod_{1\leq i\leq r}p_i^{k_i}\mid l\prod_{1\leq i\leq r}C_i$.
So the period $T(\orb^+(v,B))$ satisfies that $$\prod_{1\leq i\leq r}p_i^{k_i}\left | T(\orb^+(v,B))\prod_{1\leq i\leq r}C_i.\right.$$

For any $0\leq j<i<T(\orb^+(v,B))$, we estimate $|B^iv-B^jv\mod 1|$ as follows. Denote $v_t=p_t^{-k_t}w_t$.
If $D^iv_r-D^jv_r\not\equiv 0\mod 1$, then $|D^iv_r-D^jv_r\mod 1|^m\geq C_r'^{-1}p_r^{-k_r}$. So
$$
|B^iv-B^jv\mod 1|^n\geq |D^iv_r-D^jv_r\mod 1|^{mr}\geq
(C_r'^{-1})^r(p_r^{-k_r})^r>(C_r'^{-1}C^{-1})^r\prod_{1\leq t\leq r}p_t^{-k_t}.$$
If $D^iv_r-D^jv_r\equiv 0\mod 1$, suppose that $t$ is the maximal number such that $$\left(\sum_{s=0}^{r-t}C^{s}_iD^{i-s}v_{s+t}\right)-\left(\sum_{s=0}^{r-t}C^{s}_jD^{j-s}v_{s+t}\right)\not\equiv 0\mod 1.$$
Then we have that for any $t<l\leq r$,
$$\left(\sum_{s=0}^{r-l}C^{s}_iD^{i-s}v_{s+l}\right)-\left(\sum_{s=0}^{r-l}C^{s}_jD^{j-s}v_{s+l}\right)\equiv 0\mod 1$$
Since $(p_l^{k_l},\Pi_{s=1}^{r-l}p_{s+l}^{k_{s+l}})=1$, there are $a,b\in \ZZ$ such that
$ap_l^{k_l}+b\Pi_{s=1}^{r-l}p_{s+l}^{k_{s+l}}=1$, then
\begin{equation*}
\begin{split}
D^iv_l-D^jv_l&\equiv ap_l^{k_l}(D^iv_l-D^jv_l)+b\prod_{s=1}^{r-l}p_{s+l}^{k_{s+l}}(D^iv_l-D^jv_l)\\
&\equiv b\prod_{s=1}^{r-l}p_{s+l}^{k_{s+l}}\left(\sum_{s=0}^{r-l}C^{s}_iD^{i-s}v_{s+l}-C^{s}_jD^{j-s}v_{s+l}\right)\\
&\equiv 0\mod 1
\end{split}
\end{equation*}
which means that $T(\orb^+(v_l,D))\mid i-j$, and thus
$p_l^{k_l}\left | C_l(i-j)\right.$, $\prod_{t<l\leq r}p_l^{k_l}\left| (i-j)\prod_{t<l\leq r}C_l\right.$, i.e., there are constants $k_{l,0}$ independent of $k$, such that for sufficiently large $k$,
$$
i\equiv j\mod \prod_{t<l\leq r}p_l^{k_l-k_{l,0}}.
$$
Thus
$$
\prod_{t<l\leq r}p_l^{k_{l,0}}\left(\sum_{s=1}^{r-t}C^{s}_iD^{i-s}v_{s+t}
-\sum_{s=1}^{r-t}C^{s}_jD^{j-s}v_{s+t}\right)\equiv 0\mod 1.
$$
Denote $L=\prod_{t<l\leq r}p_l^{k_{l,0}}$. Since $D(Lw_t)\equiv b_t(Lw_t)\mod p_t^{k_t}$, using method for the irreducible case, we can prove that $|D^i(Lv_t)-D^j(Lv_t)\mod 1|^m\geq C_t'^{-1}p_t^{-k_t}$ (or equal to zero). So for sufficiently large $k$, we have that
\begin{equation*}
\begin{split}
&(L|B^iv-B^jv\mod 1|)^n\\
&\geq \left(L\left|\sum_{s=0}^{r-t}C^{s}_iD^{i-s}v_{s+t}-\sum_{s=0}^{r-t}C^{s}_jD^{j-s}v_{s+t}\mod 1\right|\right)^n\\
&\geq \left|L\left(\sum_{s=0}^{r-t}C^{s}_iD^{i-s}v_{s+t}-\sum_{s=0}^{r-t}C^{s}_jD^{j-s}v_{s+t}\right)\mod 1\right|^n \quad (\text{or} \geq 1)\\
&= |D^i(Lv_t)-D^j(Lv_t)\mod 1|^n \quad (\text{or} = 1)\\
&\geq C_t'^{-r}(p_t^{-k_t})^r\\
&\geq (C_t'^{-1}C^{-1})^r\prod_{1\leq t\leq r}p_t^{-k_t}.
\end{split}
\end{equation*}
So there is a constant $C$ such that $C\leq d(\orb^+(v,B))^nT(\orb^+(v,B))$. From the conclusion for system $B$ to system $A$ is due to Lemma \ref{similarity-matrix}.
\end{proof}

\indent We will use Structure Theorem for Modules over Polynomial Rings to simplify the reducible case of Theorem \ref{metric}, which is stated as follows before the proof of Theorem \ref{metric}.

\begin{theorem}\label{STMPR}
\cite[14.8.3]{Artin}({\rm Structure Theorem for Modules over Polynomial Rings})\\
\indent Let $R=F[t]$ be the ring of polynomials in one variable with coefficients in a field $F$.
\begin{itemize}
\item[(1)] Let $V$ be a finitely generated module over $R$. Then $V$ is a direct sum of cyclic modules
$C_1,C_2,\dots, C_k$ and a free module $L$, where $C_i$ is isomorphic to $R/(d_i)$, the elements $d_1,d_2,\dots, d_k$ are monic polynomials of positive degree, and $d_1|d_2|\dots|d_k$.
\item[(2)] The same assertion as $(1)$, except that the condition that $d_i$ divides $d_{i+1}$ is replaced by:
Each $d_i$ is a power of a monic irreducible polynomial.
\end{itemize}
\end{theorem}

\noindent{\bf Proof of Theorem \ref{metric}. }
If we consider $d(\lambda)v$ as $vd(A)$, then
there is a decomposition of $V=\QQ^n$ as $$V=\QQ[\lambda]v_1\oplus \QQ[\lambda]v_2\oplus \cdots \oplus\QQ[\lambda]v_s$$
such that $\{d(\lambda)\in\QQ[\lambda]:v_i d(A)=0\}=(d_i(\lambda))$, and $d_i$ is a power of a monic irreducible polynomial (See Theorem \ref{STMPR}(2). There is no free module here, since $f(A)=0$, where $f(x)=\det(xI-A)$). Without loss of generality, we can assume that $v_i\in\ZZ^n$ and $d_i(\lambda)\in\ZZ[\lambda]$.\\
\indent Let $$P=\left(\begin{matrix}P_1\\P_2\\\cdots\\P_s\end{matrix}\right),\quad P_i = \left(\begin{matrix}v_i\\v_iA\\\cdots \\v_iA^{r_i-1}\end{matrix}\right), \quad r_i=\deg(d_i(\lambda))$$
Then $P$ is invertible, and $PA=JP$, where $J={\rm diag}\{J_1,J_2,\cdots, J_s\}$, and the minimal polynomial of $J_i$ is $|\lambda I-J_i|=d_i(\lambda)$, $1\leq i\leq s$.\\
\indent Find $m^{(1)}_1,m^{(2)}_1,\cdots, m^{(s)}_1$ such that $(m^{(i)}_1,m^{(j)}_1)=1$ for $i\neq j$ and
for any $k$, $1<\frac{m^{(i)}_{k+1}}{m^{(i)}_k}\leq C_i$, $m^{(i)}_k$ and $m^{(i)}_{k+1}$ have the same prime factors, there are $u_i^k\in \QQ^{r_i}$ s.t. $m^{(i)}_k\mid C_i'T(\orb^+(u_i^k,J_i))$, $d(\orb^+(u_i^k,J_i))^{r_i}\geq C_i''(m^{(i)}_k)^{-1}$, where $C_i$, $C_i'$ and $C_i''$ are constants independent of $k$. Let
$v = (u_1^{k_1},u_2^{k_2},\cdots,u_s^{k_s})$, and consider $\orb^+(v,J)$.

Firstly the period $T$ of $\orb^+(v,J)$ satisfies that $\prod_{1\leq i\leq s}m^{(i)}_{k_i}\mid \prod_{1\leq i\leq s}C_i'T$, and thus
$$\prod_{1\leq i\leq s}C_i'T\geq \prod_{1\leq i\leq s}m^{(i)}_{k_i}.$$ Secondly,
$$
|J^iv-J^jv\mod 1|\geq \max_{1\leq t\leq s}|J_t^iu_t^{k_t}-J_t^ju_t^{k_t}\mod 1|,
$$
while
$$
|J_t^iu_t^{k_t}-J_t^ju_t^{k_t}\mod 1|^{r_t}=0\quad\text{or}\quad\geq C_t''(m^{(t)}_{k_t})^{-1}.$$
We choose $\{k_j\}$ as follows.
$$m^{(i+1)}_{k_{i+1}}\leq (m^{(i)}_{k_i})^{\frac{r_{i+1}}{r_i}}\leq m^{(i+1)}_{k_{i+1}+1}.$$
Then
$$
m^{(i+1)}_{k_{i+1}}\leq (m^{(i)}_{k_i})^{\frac{r_{i+1}}{r_i}}\leq C_{i+1}m^{(i+1)}_{k_{i+1}}.$$
Then for the constant $C=\prod_{1\leq i\leq s}C_i^{\frac{1}{r_i}}$ we have
$$
C^{-1}(m^{(i)}_{k_i})^{\frac{1}{r_i}}\leq (m^{(j)}_{k_j})^{\frac{1}{r_j}}\leq C(m^{(i)}_{k_i})^{\frac{1}{r_i}},$$
which means that
$$
(m^{(t)}_{k_t})^{\frac{n}{r_t}}=\prod_{1\leq j\leq s}(m^{(t)}_{k_t})^{\frac{r_j}{r_t}}\leq \prod_{1\leq j\leq s}C^{r_j}(m^{(j)}_{k_j})
=C^n\prod_{1\leq j\leq s}m^{(j)}_{k_j}.
$$
So $$d(\orb^+(v,J))^n\geq \left(\min_{1\leq t\leq s}\{(C''_t)^{\frac{1}{r_t}}\}\right)^nC^{-n}\left(\prod_{1\leq j\leq s}m^{(j)}_{k_j}\right)^{-1}.$$

Combining the inequality of $T(\orb^+(v,J))$ and $d(\orb^+(v,J))$, we get that $$T(\orb^+(v,J))d(\orb^+(v,J))^n\geq \left(\min_{1\leq t\leq s}\{(C''_t)^{\frac{1}{r_t}}\}\right)^nC^{-n}\left(\prod_{1\leq i\leq s}C_i'\right)^{-1}.$$
From the conclusion for system $J$ to system $A$ is due to Lemma \ref{similarity-matrix}.
\qed

\subsection{Proof of Proposition 1.4 and Corollary 1.5}
Before proving Proposition \ref{transfer metric to measure}, let us recall the classical theorem about the absolutely continuous measure with respect to an ergodic measure, which will be used later in the proof of Proposition \ref{transfer metric to measure}.
\begin{theorem}\label{absolutely continuous}
\cite[Page 15, Theorem 2(1)]{Sinai}
Suppose we are given a dynamical system $T$ on the measurable space $(M,\mathscr{B})$ along with two normalized measures $\mu_1, \mu_2$ defined on $\mathscr{B}$ and invariant with respect to $T$. If the measure $\mu_1$ is ergodic with respect to $T$ while $\mu_2$ is absolutely continuous with respect to $\mu_1$ ($\mu_2$ is not assumed ergodic a priori), then $\mu_1=\mu_2$.
\end{theorem}

\noindent{\bf Proof of Proposition \ref{transfer metric to measure}:}
Let $A\in M_n(\mathbb{Z})$ induce an ergodic endomorphism on an $n$-torus.
Assume that the periodic orbits $\{O_k\}_{k\in\NN}$ are uniformly distributed with constant $C>0$. Take a subsequence $\mu_{n_k}$ of $\mu_k$ such that $\mu_{n_k}\to \mu\ (k\to +\infty)$. Then $\mu$ is an invariant probability measure of $A$. In the following we firstly prove that $\mu$ is absolutely continuous with respect to Leb by the uniform distribution of $\{O_k\}_{k\in\NN}$.
According to Theorem \ref{absolutely continuous}, we know that $\mu = {\rm Leb}$, since $A$ is ergodic w.r.t Leb.

Consider the semi-ring $\mathcal{C}=\{(a,b]:a,b\in\mathbb{R}^n, b-a\in [0,1]^n\}$ on $\mathbb{T}^n$, which generates the Borel $\sigma$-algebra of $\TT^n$, i.e., $\sigma(\mathcal{C})=\mathscr{B}(\mathbb{T}^n)$. From the proof of Caratheodory extension theorem, we know that for any $B\in \mathscr{B}(\mathbb{T}^n)$,
\begin{eqnarray*}
\mu(B)&=&\inf \left\{\sum_{k=1}^{+\infty}\mu(B_k):B_k\in \mathcal{C}, B\subset \bigcup_{k=1}^{+\infty}B_k\right\}, \\
\mathrm{Leb}(B)&=&\inf \left\{\sum_{k=1}^{+\infty}{\rm Leb}(B_k):B_k\in \mathcal{C}, B\subset \bigcup_{k=1}^{+\infty}B_k\right\}.
\end{eqnarray*}
We claim that for any $R\in \mathcal{C}$,  $\mu(R)\leq C_1C^{-1}{\rm Leb}(R)$, $C_1=4^nn^{\frac{n}{2}}$. Then we can get that $\mu(B)\leq C_1C^{-1}{\rm Leb}(B)$ for any Borel measurable set $B$, which implies that $\mu$ is absolutely continuous with respect to Leb.

Now we prove the claim. Denote $$R=\prod_{i=1}^n(a_i,b_i]\subset U=\prod_{i=1}^n(a_i,c_i) \quad (b_i<c_i).$$
Divide $(a_i,c_i)$ into several parts such that each one has length $\frac{1}{2\sqrt{n}}d(O_{n_k})$. Notice that
$$\prod_{i=1}^n\left(\left[\frac{c_i-a_i}{\frac{1}{2\sqrt{n}}d(O_{n_k})}\right]+1\right)\leq \prod_{i=1}^n\frac{2(c_i-a_i)}{\frac{1}{2\sqrt{n}}d(O_{n_k})}=C_1d(O_{n_k})^{-n}{\rm Leb}(U).$$
There are at most $C_1d(O_{n_k})^{-n}{\rm Leb}(U)$ numbers of blocks. Since the diameter of each block is smaller than $\frac{1}{2}d(O_{n_k})$, every block contains at most one element of $O_{n_k}$.
So $$\mu_{n_k}(U)\leq C_1d(O_{n_k})^{-n}{\rm Leb(U)}T(O_{n_k})^{-1}\leq C_1C^{-1}{\rm Leb}(U).$$
Thus
$$\mu(R)\leq \mu(U)\leq \liminf_{k\to +\infty}\mu_{n_k}(U)\leq C_1C^{-1}{\rm Leb}(U).$$
Letting $U\to R$, we get that $\mu(R)\leq C_1C^{-1}{\rm Leb}(R)$.\\
\indent This completes the proof of the claim and Proposition \ref{transfer metric to measure}.

\qed

To finish the proof of Corollary \ref{iff}, we need the following lemma.
\begin{lemma}\label{period}
Assume that $D\in M_s(\ZZ)$, $\det(D)\neq 0$ and $\det(xI-D)$ is irreducible on $\QQ$. If there is a root $r$ of $\det(xI-D)$ which is also a root of unity, then $D$ is periodic.
\end{lemma}
\begin{proof}
Suppose that there is a root $r$ of $g(x)=\det(xI-D)$ such that $r^m=1$ but $r^k\neq 1$ for any $1\leq k<m$. Then $r^i, 0\leq i<m$ are all the roots of $x^m-1$, i.e., $r$ is a primitive $m$th root of unity. Since $g(x)$ is irreducible, $g(x)$ must equal to $\prod_{1\leq i<m,(i,m)=1}(x-r^i)$. Thus $D$ has $\varphi(m)$ numbers of different eigenvalues and all of them are primitive $m$th roots of unity. So $D$ is similar to a diagonal matrix which consists of primitive $m$th roots of unity. It implies that $D^m=I$.
\end{proof}
\begin{lemma}\label{similar}
Suppose that $\zeta$ is a root of unity, and $g(x)=x^s-d_{s-1}x^{s-1}-\cdots-d_0$ is the minimal polynomial of $\zeta$ on $\ZZ$. Let $D$ be the companion matrix of $g(x)$. If $\zeta$ is a root of $f(x)=\det(xI-A)$, then there are $P\in M_n(\ZZ)$, $\det(P)\neq 0$, $B\in M_{n-s}(\ZZ)$ and $C\in M_{(n-s)\times s}(\ZZ)$ such that $PA=\left(\begin{matrix}B&C\\0&D\end{matrix}\right)P$.
\end{lemma}
\begin{proof}
If we consider $d(x)v$ as $vd(A)$, then
there is a decomposition of $V=\QQ^n$ as $$V=\QQ[x]v_1\oplus \QQ[x]v_2\oplus \cdots \oplus\QQ[x]v_r$$
such that $\{d(x):v_i d(A)=0\}=(d_i(x))$, $d_1(x)|d_2(x)|\cdots|d_r(x)$ (See Theorem \ref{STMPR}(1)). Here $f(x)=\prod_{1\leq i\leq r}d_i(x)$. Without loss of generality, we can assume that $v_i\in\ZZ^n$ and $d_i(x)\in\ZZ[x]$.
Since $f(\zeta)=0$ and $d_1(x)|d_2(x)|\cdots|d_r(x)$, there is some $i$ such that $d_i(\zeta)=0$, thus $d_r(\zeta)=0$. There is $h(x)\in\ZZ[x]$ such that $d_r(x)=g(x)h(x)$. Find $\alpha_i\in\ZZ^n (1\leq i\leq n-s)$ such that $$\{\alpha_1,\cdots, \alpha_{n-s},v_rh(A),v_rh(A)A,\cdots,v_rh(A)A^{s-1}\}$$ is a basis of $\QQ^n$. Denote by
$$P=\left(\begin{matrix}\alpha_1\\\cdots \\ \alpha_{n-s}\\v_rh(A)\\\cdots\\v_rh(A)A^{s-1}\end{matrix}\right)$$
Then $PA=\left(\begin{matrix}B&C\\0&D\end{matrix}\right)P$.
\end{proof}

{\noindent\bf Proof of Corollary \ref{iff}.} According to Proposition \ref{transfer metric to measure}, to finish the proof of Corollary \ref{iff}, we only have to show that if $A$ is not ergodic, then there does not exist any sequence of periodic measures converging to Leb.

Suppose on the contrary that there are periodic orbits $O_n=\orb(v_n,A)$ with period $T_n$ such that $$\frac{1}{T_n}\sum_{i=0}^{T_n-1}\delta_{A^iv_n}\to {\rm Leb}\quad (n\to +\infty).$$
Since $A$ is not ergodic with respect to Leb, there are $\zeta$ and $m$ such that $f(\zeta)=0$, $\zeta^m=1$ but $\zeta^k\neq 1$ for any $1\leq k<m$, where $f(x)=\det(xI-A)$. Suppose that $g(x)=x^s-d_{s-1}x^{s-1}-\cdots-d_0$ is the minimal polynomial of $\zeta$ on $\ZZ$. By Lemma \ref{similar}, there is $P\in M_n(\ZZ)$, $\det(P)\neq 0$ such that $PA=EP$, where $E=\left(\begin{matrix}B&C\\0&D\end{matrix}\right)$. Since $\zeta$ is a primitive $m$th root of unity, by Lemma \ref{period}, $D^m=I$. Denote $$\pi:\RR^{n-s}\times\RR^s\to \{0\}\times \RR^s, \quad \pi(u,v)=v.$$ Write $$Pv_n=\left(\begin{array}{c}u_n\\ w_n\end{array}\right), \quad w_n\in\QQ^s.$$
Since the minimal period of $w_n$ under $D$ is a common factor of $m$ and $T_n$,
\begin{equation*}
\begin{split}
\frac{1}{m}\sum_{i=0}^{m-1}\delta_{D^iw_n}&=\frac{1}{T_n}\sum_{i=0}^{T_n-1}\delta_{D^iw_n}\\
&=\frac{1}{T_n}\sum_{i=0}^{T_n-1}\pi_{*}\delta_{E^iPv_n}\\
&=\frac{1}{T_n}\sum_{i=0}^{T_n-1}\pi_{*}\delta_{PA^iv_n}\\
&\to \pi_{*}P_{*}{\rm Leb}\quad(n\to +\infty).
\end{split}
\end{equation*}
But the average of $m$ dirac measures can not converge to any non-atomic measure. This contradiction finishes the proof Corollary \ref{iff}.
\qed
\begin{remark}
One may use the ``weak specification'' of ergodic endomorphism on the torus (\cite{Marcus}) to give an alternative proof for the ``only if'' part of Corollary \ref{iff}.
\end{remark}

\section{References}
\begingroup
\renewcommand{\section}[2]{}

\endgroup

\noindent Daohua Yu, School of Mathematical Sciences, Peking University, Beijing 100871, China\\
email: yudh@pku.edu.cn

\noindent Shaobo Gan, School of Mathematical Sciences, Peking University, Beijing 100871, China\\
email: gansb@pku.edu.cn

\end{document}